\documentclass[final]{amsart}
\usepackage{graphicx}
\usepackage{url}

\newcommand{\equationstar}{equation}

\newtheorem{theorem}{Theorem}
\newtheorem{lemma}{Lemma}
\newtheorem{example}{Example}
\newtheorem{definition}{Definition}
\newcommand{\iwid}{\mathrm{wid}\:}
\newcommand{\N}{\mathbb{N}}
\newcommand{\R}{\mathbb{R}}
\newcommand{\IR}{\mathbb{IR}}
\newcommand{\lb}[1]{\underline{#1}}
\newcommand{\ub}[1]{\overline{#1}}
\newcommand{\hull}{\Box}
\newcommand{\vx}{{\bf x}}
\newcommand{\vy}{{\bf y}}

\begin{document}

\title{On the Exponentiation of Interval Matrices}
\author{Alexandre Goldsztejn}
\email{alexandre.goldsztejn@univ-nantes.fr}
\urladdr{http://www.goldsztejn.com}
\address{Laboratoire d'Informatique de Nantes Atlantique\\Facult\'e des Sciences et des Techniques\\2, rue de la Houssini\`ere\\BP~92208, 44322 Nantes CEDEX 3, France}

\begin{abstract}
	The numerical computation of the exponentiation of a real matrix has been intensively studied. The main objective of a good numerical method is to deal with round-off errors and computational cost. The situation is more complicated when dealing with interval matrices exponentiation: Indeed, the main problem will now be the dependency loss of the different occurrences of the variables due to interval evaluation, which may lead to so wide enclosures that they are useless. In this paper, the problem of computing a sharp enclosure of the interval matrix exponential is proved to be NP-hard. Then the scaling and squaring method is adapted to interval matrices and shown to drastically reduce the dependency loss w.r.t. the interval evaluation of the Taylor series.
\end{abstract}

\maketitle


\section{Introduction}

The exponentiation of a real matrix allows solving initial value problems (IVPs) for linear ordinary differential equations (ODEs): given $A\in\R^{n\times n}$, the solution of the IVP defined by $\vy'(t)=A \  \vy(t)$ and $\vy(0)=\vy_0$ is $\vy(t)=\exp(tA) \  \vy_0$, where for any $M\in\R^{n\times n}$
\begin{equation}\label{eq:real-exp}
	\exp(M):=\sum_{k=0}^{\infty} \frac{M^k}{k!}.
\end{equation}
Linear ODE being met in many contexts, the numerical computation of the matrix exponential has been intensively studied (see \cite{Ward1977} ,\cite{Bochev:1989:SVN}, \cite{Moler2003}, \cite{Higham2005} and references therein). While an approximate computation of \eqref{eq:real-exp} leads to an approximate solution for the underliying IVP, interval analysis (see Section~\ref{s:IA}) offers a more rigorous framework: In most practical situations the parameters that define the linear ODE are known with some uncertainty. In this situation, one usually ends with an interval of matrices $[A]=[\lb{A},\ub{A}]:=\{A\in\R^{n\times n}:\lb{A}\leq A\leq\ub{A}\}$ inside which the actual matrix $A$ is known to be. Then, the rigorous enclosure of the solution will be obtained computing an interval matrix that encloses the exponentiation of the interval matrix $[A]$:
\begin{equation}\label{eq:interval-exp}
	\exp([A]):=\{\exp(A):A\in[A]\}.
\end{equation}

The most obvious way of obtaining an interval enclosure of $\exp(A)$ is to evaluate the truncated Taylor series using interval arithmetic and to bound the remainder (cf. Subsection~\ref{ss:Taylor-series} for details). However, the next example shows that the truncated Taylor series is not well adapted to interval evaluation, even if no truncation of the series is performed.
\begin{example}\label{ex:correlation-loss}
	Consider the interval of matrices $A:=[\lb{A},\ub{A}]$ where
	\begin{\equationstar}
		\lb{A}:=\begin{pmatrix}0 & 1\\0 & -3\end{pmatrix} \ \ , \ \ \ub{A}:=\begin{pmatrix}0 & 1\\0 & -2\end{pmatrix} \ \ \text{and} \ \ A(t):=\begin{pmatrix}0 & 1\\0 & t\end{pmatrix}.
	\end{\equationstar}
	Computing the formal expression of the exponential of the matrix $A(t)$ for $t\in[-3,-2]$, it can be proved that $\lb{X}\leq\exp([A])\leq\ub{X}$ with
	\begin{equation}\label{eq:interval-hull-exp}
	\begin{array}{ccccc}
		\lb{X} & = & \begin{pmatrix}1 & \frac{1}{3} \left(1-e^{-3}\right) \\ 0 & e^{-3}\end{pmatrix} & \approx & \begin{pmatrix}1 & 0.316738 \\0 & 0.0497871\end{pmatrix}
		\\ \ub{X} & = & \begin{pmatrix}1 & \frac{1}{2} \left(1-e^{-2}\right) \\ 0 & e^{-2}\end{pmatrix} & \approx & \begin{pmatrix}1 & 0.432332 \\ 0 & 0.135335\end{pmatrix},
	\end{array}
	\end{equation}
	where the lower and upper bounds cannot be improved, i.e. $[\lb{X},\ub{X}]$ is the optimal enclosure of $\exp([A])$.
	
	Now, computing the interval Taylor series with an order of $10$ (which is high enough to make the remainder insignificant) leads to $\lb{T}\leq\exp([A])\leq\ub{T}$ with
	\begin{\equationstar}
		\lb{T}=\begin{pmatrix}1 & -1.20912 \\0 & -6.25568\end{pmatrix} \ \ \text{and} \ \ \ub{T}=\begin{pmatrix}1 & 1.95819 \\ 0 & 6.4408\end{pmatrix}.
	\end{\equationstar}
	This is actually an enclosure of \eqref{eq:interval-hull-exp}, but a very pessimistic one.
\end{example}

As shown by the previous example, even with high enough order for the expansion so the influence of the remainder is insignificant, the interval evaluation of the Taylor series computes very crude bounds on the exponential of an interval matrix. The reason of this bad behavior of the Taylor series interval evaluation is the dependency loss between the different occurrences of variable that occurs during the interval evaluation of an expression (cf. Section \ref{ss:dependency}). In general, one cannot expect to compute the optimal enclosure of \eqref{eq:interval-exp}: The NP-hardness of this problem is proved in Section \ref{s:NPH}.

Two well known techniques can help decreasing the pessimism of the interval evaluation: First, centered forms can give rise to sharper enclosures than the natural interval evaluation for small enough interval inputs. Such a centered form for the matrix exponential was proposed in \cite{Oppenheimer:1988:AIAa,Oppenheimer:1988:AIAb,Oppenheimer:1988:AIAc}. However, this centered evaluation dedicated to the interval matrix exponentiation is quite complex and very difficult to follow or implement. Furthermore, there is an error in the proof of Proposition~10M of \cite{Oppenheimer:1988:AIAa}\footnote{The proof of Proposition Proposition~10M of \cite{Oppenheimer:1988:AIAa} is claimed to be similar to the proof of Proposition~10 of \cite{Oppenheimer:1988:AIAa}. However, the proof of Proposition~10 uses the fact that $f([x])=\cup_{x\in[x]}f(x)$, which is valid only for scalar functions but not for vector-valued or matrix-valued functions, and thus cannot be extended to prove Proposition~10M which involves matrix-valued functions.} and some non justified assumptions in Section~VII \cite{Oppenheimer:1988:AIAb}.

The second technique consists of formally rewriting the expression so as to obtain a formula more suited to interval evaluation (usually decreasing the number of occurrences of variables). For example, the evaluation of a polynomial in its Horner form is known to improve its interval evaluation \cite{Ceberio2002}. It can be naturally applied to the Taylor series of the matrix exponential and was actually used in \cite{Oppenheimer:1988:AIAb} to exponentiate the center matrix as required in the centered form. A proof of the correctness of the matrix exponential Taylor series Horner evaluation with rigorous bound on the truncation much simpler than the one given in \cite{Oppenheimer:1988:AIAb} is provided in Subsection~\ref{ss:Horner}. In Subsection \ref{ss:scaling-squaring}, we extend the well known scaling and squaring process (which consists of rewriting the Taylor series using the formula $\exp M=(\exp M/2^s)^{2^s}$) to the exponentiation of interval of matrices. In addition of the usual benefits of this process, its use in conjunction with interval analysis allows an automatic control of the rounding errors. Furthermore, it is shown to drastically improve the dependency loss due to the interval evaluation, hence providing much more accurate and less expensive computations. As explained in Section~\ref{s:experimentations} dedicated to experiments, the enclosure formula based on the scaling and squaring process is not only much simpler than the centered evaluation proposed in~\cite{Oppenheimer:1988:AIAb} but it also provides sharper enclosures.


\section{Interval Analysis}\label{s:IA}

Interval analysis (IA) is a modern branch of numerical analysis that was born in the 60's. It consists of computing with intervals of reals instead of reals, providing a framework for handling uncertainties and verified computations (see \cite{Moore1966,Alefeld1974,Neum90,Jaulin2001} and \cite{Kearfott1996-2} for a survey).

\subsection{Intervals, interval vectors and interval matrices}

An interval is a connected subset of $\R$. Intervals are denoted by bracketed symbols, e.g. $[x]\subseteq \R$. When no confusion is possible, lower an upper bounds of an interval $[x]$ are denoted by $\lb{x}\in\R$ and $\ub{x}\in\R$, with $\lb{x}\leq\ub{x}$, i.e. $[x]=[\lb{x},\ub{x}]=\{x\in\R:\lb{x}\leq x\leq \ub{x}\}$. Furthermore, a real number $x$ will be identified with the degenerated interval $[x,x]$.

There are two equivalent ways of defining interval matrices. On the one hand, being given two matrices $\lb{A}\leq\ub{A}\in\R^{n\times m}$ (where the inequality is defined componentwise), an interval of matrices is obtained by considering
\begin{equation}
	[A]:=\{A\in\R^{n\times m}:\lb{A}\leq A\leq \ub{A}\}.
\end{equation}
On the other hand, being given intervals $[a_{ij}]$, a matrix of intervals is obtained by considering
\begin{equation}
	[A]:=\{A\in\R^{n\times m}:\forall i\in\{1,\ldots n\},\forall j\in\{1,\ldots m\},a_{ij}\in[a_{ij}]\}.
\end{equation}
These two definitions are obviously equivalent following the notational convention $\lb{A}=(\lb{a}_{ij})$, $\ub{A}=(\ub{a}_{ij})$ and $[a_{ij}]=[\lb{a}_{ij},\ub{a}_{ij}]$, and will be used indifferently.

Interval vectors are defined similarly to interval matrices as either intervals of vectors or vectors of intervals.

The magnitude of an interval $[x]$ is $|[x]|:=\max\{|\lb{x}|,|\ub{x}|\}$. The magnitude of an interval matrix is the real matrix formed of the magnitude of the entries of the interval matrix, i.e. $|[A]|:=(|[a_{ij}]|)_{ij}$. The infinite norm will be considered in the rest of the paper. The norm of an interval matrix is the maximum of the norms of the real matrices included in this interval matrix. It is easily computed as $||[A]||=|| \,|[A]|\,||$. It obviously satisfies $[A]\subseteq[B]$ implies $||[A]||\leq||[B]||$.

The interval hull of a set of real is the smallest interval that contains this set. It is denoted by $\hull$. For example $\hull\{0,1\}=[0,1]$. The interval hull is defined similarly for sets of vectors and sets of matrices.

\subsection{Interval Arithmetic}
Operations $\circ\in\{+,\times,-,\div\}$ are extended to intervals in the following way:
\begin{equation}\label{def:IA}
	[\lb{x},\ub{x}]\circ[\lb{y},\ub{y}]:=\{x\circ y:x\in[\lb{x},\ub{x}],y\in[\lb{y},\ub{y}]\}.
\end{equation}
The division is defined for intervals $[\lb{y},\ub{y}]$ that do not contain zero. Note that unary elementary functions like $\exp$, $\ln$, $\sin$, etc., can also be extended to intervals similarly. All these elementary interval extension form the interval arithmetic (IA). As real numbers are identified to degenerated intervals, the IA actually generalizes the real arithmetic, and mixed operations like $1+[1,2]=[2,3]$ are interpreted using \eqref{def:IA}.

The IA lacks some important properties verified by its real counterpart: It is not a field anymore as interval addition and interval multiplication has no inverse in general, while distributivity is not valid anymore (instead a subdistributivity law holds in the form $[x]([y]+[z])\subseteq [x]\,[y]+[x]\,[z]$). On the other hand, interval operations are inclusion-increasing, i.e. $[x]\subseteq[x']$ and $[y]\subseteq[y']$ imply $[x]\circ[y]\subseteq [x']\circ[y']$.

\paragraph{Rounded Computations}

As real numbers are approximately represented by floating point numbers \cite{Goldberg91}, the IA cannot match the definition \eqref{def:IA} exactly. In order to preserve the inclusion property, the IA has to be implemented using an outward rounding. For example, $[1,3]/[2,2]=[0.5,1.5]$ while both $0.5$ and $1.5$ cannot be exactly represented with floating point numbers. Therefore, the computed result will be $[0.5^-,1.5^+]$ where $0.5^-$ (respectively $1.5^+$) is a floating point number smaller than $0.5$ (respectively bigger than $1.5$). Of course, a good implementation will return the greatest floating point number smaller than $0.5$ and the smallest floating point number greater than $1.5$. Among other implementations of IA, we can cite the C/C++ libraries PROFIL/BIAS \cite{Knueppel1994} and Gaol \cite{goualard:gaol_manual}, the Matlab toolbox INTLAB \cite{Ru99a} and Mathematica \cite{Mathematica}.

\subsection{Interval Evaluation of an Expression}

The natural usage of the IA is to evaluate an expression for interval arguments. The fundamental theorem of interval analysis (cf. \cite{Moore1966}) allows explaining the interpretation of this interval evaluation. Its proof is classical but is reproduced here.
\begin{theorem}
	Let $\mathbb{E}$ and $\mathbb{F}$ be either $\R$ or $\R^n$ or $\R^{n\times n}$ and $[x]\in\mathbb{IE}$. Consider a real function $f:[x]\longrightarrow \mathbb{F}$ and an interval function $[f]:\mathbb{I}[x]\longrightarrow \mathbb{IF}$, where $\mathbb{I}[x]$ is the set of all intervals included in $[x]$, i.e. $\mathbb{I}[x]=\{[y]\in\mathbb{IE}:[y]\subseteq[x]\}$. Suppose furthermore that both
	\begin{itemize}
		\item[(1)] For all $x\in[x]$, $f(x)\in [f](x)$
		\item[(2)] $[f]$ is inclusion-increasing in $\mathbb{I}[x]$.
	\end{itemize}
	Then, $[f]([x])\supseteq \{f(x):x\in[x]\}$.
\end{theorem}
\begin{proof}
	For all $x\in[x]$ we have $f(x)\in [f](x)$ by $(1)$, and because $[x,x]\subseteq[x]$, $(2)$ implies $[f](x)\subseteq [f]([x])$. $\Box$
\end{proof}

Let us illustrate the usage of the fundamental theorem of IA on a simple example, which can be trivially generalized to arbitrary expressions. Consider the expression $x+xy$. When evaluated for degenerated interval arguments, it gives rise to $[x,x]+[x,x]\,[y,y]=[x+xy,x+xy]\ni x+xy$. Furthermore, it is inclusion-increasing as it is compound of inclusion-increasing operations. Therefore, the fundamental theorem of IA proves that $[x]+[x]\,[y]\supseteq \{x+xy:x\in[x],y\in[y]\}$. As another example, the interval evaluation of the expression of the interval matrix/matrix product $[A] \, [B]$
\begin{equation}\label{eq:matrix-matrix-product}
	\bigl([A] \, [B]\bigr)_{ij} \ = \ \sum_k \ [a_{ik}]\,[b_{kj}]
\end{equation}
gives rise to the inclusion $[A] \, [B]\supseteq \{AB:A\in[A],B\in[B]\}$.

When the expression evaluated for interval arguments contains only one occurrence of each variable, the computed enclosure of the range is optimal. As a consequence, the expression \eqref{eq:matrix-matrix-product} is optimal since only one occurrence of each variable is involved in the expression of each entry. Thus the more accurate statement $[A] \, [B]=\hull \{AB:A\in[A],B\in[B]\}$ actually holds. Note that $[A] \, [B]\neq \{AB:A\in[A],B\in[B]\}$ since the product $[A] \, [B]$ actually contains matrices that are not the product of matrices from $[A]$ and $[B]$, but $[A] \, [B]$ is the smallest interval matrix that contains $\{AB:A\in[A],B\in[B]\}$.

However, the interval evaluation of expression that contains several occurrences of some variable is not optimal anymore in general. In this case, some overestimation generally occurs which can dramatically decrease the usefulness of interval evaluation.

\subsection{Overestimation of Interval Evaluation}\label{ss:dependency}

When an expression contains several occurrences of some variables its interval evaluation generally gives rise to a pessimistic enclosure of the range. For example, the evaluation of $x+xy$ for the arguments $[x]=[0,1]$ and $[y]=[-1,0]$ gives rise to the enclosure $[-1,1]$ of $\{x+xy:x\in[x],y\in[y]\}$ while the evaluation of $x(1+y)$ for the same interval arguments gives rise to the better enclosure $[0,1]$ of the same range (the latter enclosure being optimal since the expression $x(1+y)$ contains only one occurrence of each variables). This overestimation is the consequence of the loss of correlation between different occurrences of the same variables when the expression is evaluated for interval arguments.

While $[A] \, [B]=\hull\{AB:A\in[A],B\in[B]\}$, the interval evaluation of $[A]\,[A]$, which encloses $\{A^2:A\in[A]\}$, is not optimal in general since several occurrences of some entries of $[A]$ appear in each expression of the entries of $[A]\,[A]$. An algorithm for the computation of $\hull \{A^2:A\in[A]\}$ which can be evaluated with a number of interval operations that is polynomial w.r.t. the dimension of $[A]$ was proposed in \cite{Kreinovich-SAC2005}. However, it was proved that no such polynomial algorithm exists for the computation of $\hull \{A^3:A\in[A]\}$ unless P=NP, i.e. the computation of $\hull \{A^3:A\in[A]\}$ is NP-hard (cf. \cite{Kreinovich-SAC2005}).

The situation is even worth than this: Even computing an enclosure of $\{A^3:A\in[A]\}$ for a fixed precision is NP-hard. The notion of $\epsilon$-accuracy of an enclosure is introduced to formalize this problem (see e.g. \cite{Gaganov1985,Krei98}). The following definition is adapted to sets of matrices.
\begin{definition}
	Let $\mathbb{A}\subseteq\mathbb \R^{n\times m}$ be a set of matrices, $[A]=\hull\mathbb A\in\IR^{n\times m}$, and consider an interval enclosure $[B]$ of $\mathbb A$ (which obviously satisfies $[B]\supseteq [A]$). The interval enclosure $[B]$ is said $\epsilon$-accurate if
	\begin{equation}
		\max\bigl\{ \ \max_{ij}|\lb{a}_{ij}-\lb{b}_{ij}| \ , \ \max_{ij}|\ub{a}_{ij}-\ub{b}_{ij}| \ \bigr\} \ \leq \ \epsilon
	\end{equation}
\end{definition}

Thus, an $\epsilon$-accuracy enclosure of a set of matrices is $\epsilon$-accurate for each entry. Although it is not stated in \cite{Kreinovich-SAC2005}, the proof presented there also shows that the computation of an $\epsilon$-accuracy enclosure of $\{A^3:A\in[A]\}$ is NP-hard.


\section{Computational Complexity of the $\epsilon$-Accurate Interval Matrix Exponentiation}\label{s:NPH}

Computing $\epsilon$-accurate interval enclosures of the range of a multivariate polynomial $f:\R^n\longrightarrow\R$ is NP-hard (cf. \cite{Gaganov1985} and Theorem 3.1 in \cite{Krei98}). Even if one restricts its attention to bilinear functions, the computation of $\epsilon$-accurate enclosures of their range remains NP-hard (cf. Theorem 5.5 in \cite{Krei98}). Note that if one fixes the dimension of the problems, then the computation of these $\epsilon$-accurate enclosures is not NP-hard anymore, hence showing that the NP-hardness is linked to the growth of the problem dimension. It is not a surprise that computing an $\epsilon$-accurate enclosure of the interval matrix exponential is NP-hard, but this result remains to be proved.

\begin{theorem}
	For every $\epsilon>0$, computing an $\epsilon$-accurate enclosure of $\exp([A])$ is NP-hard.
\end{theorem}
\begin{proof}
	We prove that the $\epsilon$-accurate enclosure of the range of a bilinear function, which is NP-hard by Theorem 5.5 in \cite{Krei98}, reduces to the $\epsilon$-accurate enclosure of the interval matrix exponential. Let $B\in\R^{n\times n}$ and $[\vx],[\vy]\in\IR^n$. Define $\lb{A},\ub{A}\in\R^{(2n+2)\times(2n+2)}$ by
	\begin{\equationstar}
		\lb{A}:=\left(\begin{array}{c|c|c|c}
			0 & \lb{\vx}^T & 0 & 0
			\\\hline 0 & 0 & B & 0
			\\\hline 0 & 0 & 0 & \lb{\vy}
			\\\hline 0 & 0 & 0 & 0
		\end{array}\right)
		\ \text{and} \ 
		\ub{A}:=\left(\begin{array}{c|c|c|c}
			0 & \ub{\vx}^T & 0 & 0
			\\\hline 0 & 0 & B & 0
			\\\hline 0 & 0 & 0 & \ub{\vy}
			\\\hline 0 & 0 & 0 & 0
		\end{array}\right),
	\end{\equationstar}
	which are obviously computed in polynomial time from $B$, $[\vx]$ and $[\vy]$. We now prove that an $\epsilon$-accurate enclosure of the exponentiation of $[A]$ gives rise to an $\epsilon$-accurate enclosure of the range of the image of $[\vx]$ and $[\vy]$ by the function $\vx^T \  B \ \vy$, which will conclude the proof. Let
	\begin{\equationstar}
		A:=\left(\begin{array}{c|c|c|c}
			0 & \vx^T & 0 & 0
			\\\hline 0 & 0 & B & 0
			\\\hline 0 & 0 & 0 & \vy
			\\\hline 0 & 0 & 0 & 0
		\end{array}\right)
	\end{\equationstar}
	be such that $A\in[A]$, that is equivalently $\vx\in[\vx]$ and $\vy\in[\vy]$. One can check easily that $A$ is nilpotent:
	\begin{\equationstar}
		A^2=\left(\begin{array}{c|c|c|c}
			0 & 0 & \vx^T B & 0
			\\\hline 0 & 0 & 0 & B\vy
			\\\hline 0 & 0 & 0 & 0
			\\\hline 0 & 0 & 0 & 0
		\end{array}\right)
		\ \ \
		A^2=\left(\begin{array}{c|c|c|c}
			0 & 0 & 0 & \vx^TB\vy
			\\\hline 0 & 0 & 0 & 0
			\\\hline 0 & 0 & 0 & 0
			\\\hline 0 & 0 & 0 & 0
		\end{array}\right)
		\ \ \ A^3=0.
	\end{\equationstar}
	Thus $\bigl(\exp(A)\bigr)_{1,2n+2}=\frac{1}{6} \ \vx^T \  B \ \vy$. As a consequence,
	\begin{\equationstar}
		\bigl\{\bigl(\exp(A)\bigr)_{1,2n+2}:A\in[A]\bigr\}=\{\vx^T \  B \ \vy:\vx\in[\vx],\vy\in[\vx]\}
	\end{\equationstar}
	and the entry $(1,2n+2)$ of an $\epsilon$-accurate enclosure of $\exp([A])$ is an $\epsilon$-accurate enclosure of the image of $[\vx]$ and $[\vy]$ by the function $\vx^T \  B \ \vy$. $\Box$
\end{proof}


\section{Polynomial Time Algorithms for the Enclosure of the Interval Matrix Exponential}\label{s:scaling-squaring}

This section presents three expressions dedicated to the enclosure of the exponential of an interval matrix: The naive interval evaluation of the Taylor series, the interval evaluation of the Taylor series following the Horner scheme and the interval evaluation of the series following the scaling ans squaring process.

\subsection{Taylor Series}\label{ss:Taylor-series}

The naive interval evaluation of the truncated Taylor series for interval matrices exponential is now presented including a rigorous bound on the truncation error. The bound used here is the same as in in \cite{Oppenheimer:1988:AIAb}. Let us define for  $K+2>||[A]||$
\begin{equation}\label{eq:interval-Taylor}
\begin{array}{rl}
	{[\tilde{\mathcal{T}}]}([A],K) := & I+[A]+\frac{1}{2}[A]^2+\ldots+\frac{1}{K!}[A]^K
	\\ {[\mathcal{T}]}([A],K) := & [\tilde{\mathcal{T}}]([A],K)+[\mathcal{R}]([A],K),
\end{array}
\end{equation}
where the interval remainder $[\mathcal{R}]([A],K)$ is
\begin{equation}\label{eq:remainder}
[\mathcal{R}]([A],K):=\rho(||[A]||,K) \; [-E,E] \ \ \text{with} \ \ \rho(\alpha,K)=\frac{\alpha^{K+1}}{(K+1)!\;\bigl(1-\frac{\alpha}{K+2}\bigr)}
\end{equation}
and $E\in\R^{n\times n}$ has all its entries equal to $1$. We provide now a new proof that ${[\mathcal{T}]}([A],K)$ is an enclosure of $\{\exp(A):A\in[A]\}$ which is much simpler than the one provided in~\cite{Oppenheimer:1988:AIAb} and which will be used in the rest of the paper. The following lemmas will allow applying the fundamental theorem of interval analysis to expressions that include $[\mathcal{R}](\,.\,,K)$.
\begin{lemma}\label{lem:R}
	For a fixed positive integer $K$, the interval matrix operator $[\mathcal{R}](\,.\,,K)$ is inclusion-increasing inside $\{[A]\in\IR^{n\times n}:||[A]||<K+2\}$.
\end{lemma}
\begin{proof}
	Let $[A],[B]\in\IR^{n\times n}$ such that $||[B]||\leq K+2$ and $[A]\subseteq[B]$. Then, $||[A]||\leq||[B]||$ which implies $||[A]||\leq K+2$. Furthermore as $\rho(\alpha,K)$ is obviously increasing with respect to $\alpha$, we have $\rho(||[A]||,K)\leq\rho(||[B]||,K)$. Finally, as $[-E,E]$ is centered on the null matrix, we have
	\begin{\equationstar}
		\rho(||[A]||,K)\;[-E,E]\subseteq \rho(||[B]||,K)\;[-E,E],
	\end{\equationstar}
	which concludes the proof. $\Box$
\end{proof}

The next lemma a direct consequence the well known upper bound on the truncation error of the exponential series.

\begin{lemma}\label{lem:Taylor}
	Let $A\in\R^{n\times n}$ and $K\in\N$ such that $K+2>||A||$. Then $\exp(A) \in [\mathcal{T}](A,K)$.
\end{lemma}
\begin{proof}
	Suppose that $\exp(A) \notin [\mathcal{T}](A,K)$, i.e. there exist $i,j\in\{1,\ldots n\}$ such that 
	\begin{\equationstar}
		\bigl(\exp(A)\bigr)_{ij}\notin \Bigl(\sum_{k=0}^{K}\frac{A^k}{k!}\Bigr)_{ij}+ \ \rho(||A||,K)\;[-1,1].
	\end{\equationstar}
	This obviously implies
	\begin{\equationstar}
		\bigl|\bigl(\exp(A)\bigl)_{ij}-\bigr(\sum_{k=0}^{K}\frac{A^k}{k!}\bigr)_{ij}\bigr|> \rho(||A||,K).
	\end{\equationstar}
	Therefore $||\exp(A)-\sum_{k=0}^{K}\frac{A^k}{k!}||> \rho(||A||,K)$ holds, which contradicts the well known bound on the truncation error for the exponential series (see e.g. \cite{Oppenheimer:1988:AIAb}). Eventually $\exp(A) \in [\mathcal{T}](A,K)$ has to hold. $\Box$
\end{proof}

Theorem \ref{thm:Taylor} below states that ${[\mathcal{T}]}([A],K)$ is an enclosure of $\exp([A])$. It was stated in \cite{Oppenheimer:1988:AIAb} but proved with different arguments in \cite{Oppenheimer:1988:AIAa}, \cite{Oppenheimer:1988:AIAb}. Note that the usage of the fundamental theorem of interval analysis allows us to provide a proof much simpler than the one proposed in \cite{Oppenheimer:1988:AIAa}, \cite{Oppenheimer:1988:AIAb}.
\begin{theorem}\label{thm:Taylor}
	Let $[A]\in\IR^{n\times n}$ and $K\in\N$ such that $K+2>||[A]||$. Then $\exp([A]) \subseteq [\mathcal{T}]([A],K)$.
\end{theorem}
\begin{proof}
	First, by Lemma \ref{lem:R} $[\mathcal{R}](\;.\;,K)$ is inclusion-increasing, and therefore so is $[\mathcal{T}](\;.\;,K)$ because it is compounded of inclusion-increasing operators. Second, by Lemma \ref{lem:Taylor} $(\forall A\in[A]) \ \exp(A)\in[\mathcal{T}](A,K)$. Therefore, one can apply the fundamental theorem of interval analysis to conclude the proof. $\Box$
\end{proof}

\begin{example}\label{ex:correlation-loss-Taylor}
	Consider the interval of matrices $[A]$ defined in Example \ref{ex:correlation-loss}. Theorem \ref{thm:Taylor} with $K=16$ gives rise to the following enclosure of $\exp([A])$:
	\begin{equation}\label{eq:ex:correlation-loss-Taylor}
		\begin{pmatrix} {1+[-9\times 10^{-7}, 9\times 10^{-7}]} & {[-1.2092, 1.9582]} \\ {[-9\times 10^{-7}, 9\times 10^{-7}]} & {[-6.2557, 6.4409]} \end{pmatrix}.
	\end{equation}
	Higher order for the expansion do not improve the entries $(1,2)$ and $(2,2)$ anymore.
\end{example}

\subsection{Horner scheme}\label{ss:Horner}

The Horner evaluation of a real polynomial improves both the computation cost and the stability (see e.g. \cite{Knuth1997}). When an interval evaluation is computed, the Horner evaluation furthermore improves the effect of the loss of correlation (see \cite{Ceberio2002}). It is therefore natural to evaluate \eqref{eq:interval-Taylor} using a Horner scheme:
\begin{equation}\label{eq:interval-Horner}
\begin{array}{rl}
[\tilde{\mathcal{H}}]([A],K) := &  I+ [A]\Bigl(I+\frac{[A]}{2}\Bigl(I+\frac{[A]}{3}\Bigl( \ \  \cdots \ \ \Bigl(I+\frac{[A]}{K}\Bigr) \cdots\Bigr)\Bigr)\Bigr)
\\ {[\mathcal{H}]}([A],K) := & [\tilde{\mathcal{H}}]([A],K) + [\mathcal{R}]([A],K).
\end{array}
\end{equation}
\begin{lemma}\label{lem:Horner}
	Let $A\in\R^{n\times n}$ and $K\in\N$ such that $K+2>||A||$. Then $\exp(A) \in [\mathcal{H}](A,K)$.
\end{lemma}
\begin{proof}
	As interval operations are evaluated with real arguments, the Horner scheme can be expanded exactly leading to $[\tilde{\mathcal{H}}](A,K)=[\tilde{\mathcal{T}}](A,K)$. As a consequence, $[\mathcal{H}](A,K)=[\mathcal{T}](A,K)$ and Lemma \ref{lem:Taylor} concludes the proof. $\Box$
\end{proof}
\begin{theorem}\label{thm:Horner}
	Let $[A]\in\IR^{n\times n}$ and $K\in\N$ such that $K+2>||[A]||$. Then $\exp([A]) \in [\mathcal{H}]([A],K)$.
\end{theorem}
\begin{proof}
	As a consequence of Lemma \ref{lem:R}, $[H](\,.\,,K)$ is compounded of inclusion-increasing operator and is hence inclusion increasing while Lemma \ref{lem:Horner} shows that $\exp(A) \in [\mathcal{H}](A,K)$. Therefore, one can use the fundamental theorem of interval analysis to conclude the proof. $\Box$
\end{proof}

\begin{example}\label{ex:correlation-loss-Horner}
	Consider the interval of matrices $[A]$ defined in Example \ref{ex:correlation-loss}. Theorem \ref{thm:Horner} with $K=16$ then to the following enclosure of $\exp([A])$:
	\begin{\equationstar}
		\begin{pmatrix} {1+[-1.1\times 10^{-6}, 1.1\times 10^{-6}]} & {[-0.0706, 0.7352]} \\ {[-1.1\times 10^{-6}, 1.1\times 10^{-6}]} & {[-1.2056, 1.2117]} \end{pmatrix}.
	\end{\equationstar}
	This enclosure is sharper than the one computed using the Taylor series: as it was forseen, the Horner evaluation actually improves the loss of dependency introduced by the interval evaluation in the expression of the Taylor expansion of the matrix exponential.
\end{example}

\subsection{Scaling and squaring process}\label{ss:scaling-squaring}

The scaling and squaring process is one of the most efficient way to compute a real matrix exponential. It consists of first computing $\exp(A/2^L)$ and then squaring $L$ times the resulting matrix:
\begin{\equationstar}
	\exp(A)=\bigl(\exp(A/2^L)\bigr)^{2^L}.
\end{\equationstar}
Therefore, one first has to compute $\exp(A/2^L)$. This computation is actually much easier than $\exp(A)$ because $||A/2^L||$ can be made much smaller than $1$. Usually, Pad\'e approximations are used to compute $\exp(A/2^L)$. However, this technique has not been extended to interval matrices, hence we propose here to use the Horner evaluation of the Taylor series instead. Therefore, we propose the following operator for the enclosure of an interval matrix exponential: Let $K$ and $L$ be such that $(K+2)2^L>||[A]||$ and define

\begin{equation}\label{eq:interval-SS}
[\mathcal{S}]([A],L,K):=\Bigl([\mathcal{H}]\bigl([A]/2^L,K\bigr)\Bigr)^{2^L}.
\end{equation}
The exponentiation in \eqref{eq:interval-SS} is of course computed with $L$ successive interval matrix square operations.

\begin{theorem}\label{thm:SS}
	Let $[A]\in\IR^{n\times n}$ and $K,L\in\N$ such that $(K+2)\,2^L>||[A]||$. Then $\exp([A]) \subseteq [\mathcal{S}]([A],L,K)$.
\end{theorem}
\begin{proof}
	By Theorem~\ref{thm:Horner}, we have $\exp(A/2^L)\in[\mathcal{H}]\bigl([A]/2^L,K\bigr)$ for an arbitrary $A\in[A]$. The interval evaluation $[M]^{2^L}$ of an arbitrary interval matrix $[M]$ encloses $\{M^{2^L}:M\in[M]\}$, and therefore $[\mathcal{S}]([A],L,K)\ni\exp(A/2^L)^{2^L}$. This concludes the proof as this holds for an arbitrary $A\in[A]$. $\Box$
\end{proof}


\begin{example}\label{ex:correlation-loss-SS}
	Consider the interval of matrices $[A]$ defined in Example \ref{ex:correlation-loss}. Theorem \ref{thm:SS} with $L=10$ and $K=10$ leads to the following enclosure of $\exp([A])$:
	\begin{\equationstar}
		\begin{pmatrix} {1+[-5.7 \times 10^{-13}, 9.1 \times 10^{-13}]} & {[0.3165, 0.4325]} \\ {[-2.4 \times 10^{-19}, 2.4 \times 10^{-19}]} & {[0.0496, 0.1355]} \end{pmatrix}.
	\end{\equationstar}
	This enclosure is much sharper than the two previously computed using the Taylor series (cf. Example \ref{ex:correlation-loss-Taylor}) and its Horner evaluation (cf. Example \ref{ex:correlation-loss-Horner}). It is also very close to the optimal enclosure~\eqref{eq:interval-hull-exp}. The computation cost for $L$ and $K$ is approximately the same as the Horner scheme with order $L+K$.
\end{example}


\section{Experiments}\label{s:experimentations}

In this section, we compare the interval Horner evaluation of the truncated Taylor series versus the interval scaling and squaring method. The direct Taylor series is not presented as it is similar but always worth that its Horner evaluation. In order to compare these two enclosures, we use the width of these interval enclosure: Let $\iwid[M]$ be the real matrix formed of the widths of the entries of $[M]$. We will use the $||\iwid[M]||$ as a quality measure of the enclosure $[M]$. Experimentations have been carried out using Mathematica \cite{Mathematica}. Subsection \ref{ss:detailed1} presents a detailed study of the exponentiation of one real matrix, while Subsection \ref{ss:detailed2} deal with an interval matrix exponentiation.

As explained in introduction, the comparisons presented in this section do not include the interval matrix enclosure method proposed in \cite{Oppenheimer:1988:AIAa,Oppenheimer:1988:AIAb,Oppenheimer:1988:AIAc}. However, this method is based on a interval Horner evaluation and thus cannot give rise to better enclosures than the interval Horner evaluation of the center matrix, which is of poor quality as demonstrated above.

\subsection{Interval exponentiation of a real matrix}\label{ss:detailed1}

In this section, we consider the matrix $A$ defined by
\begin{\equationstar}
	A:=\begin{pmatrix}-131& 19& 18 \\ -390& 56& 54 \\ -387& 57& 52\end{pmatrix}
\end{\equationstar}
proposed in \cite{Bochev:1989:SVN}. This matrix is difficult to exponentiate because it has significant eigenvalue separation and a poorly conditioned eigenvector set.

The matrix $A$ is too difficult to exponentiate using the Horner evaluation of the truncated interval Taylor series: As $||A||=500$ the Taylor series requires an expansion of order greater than $502$. Computing $[H](A,K)$ using double precision does not provide any meaningful enclosure. Figure~\ref{fig:fig1} shows the quality of the enclosure obtained for different orders ranging from $1500$ to $1800$ and two different precisions for computations (the Mathematica \cite{Mathematica} arbitrary precision interval arithmetic was used). It shows that no meaningful enclosure is obtained for precision less than $p=110$ digits or order less than $K=1550$ using the Horner interval evaluation of the Taylor expansion.

    

\begin{figure}[t!]
	\begin{center}
		\includegraphics[scale=.8]{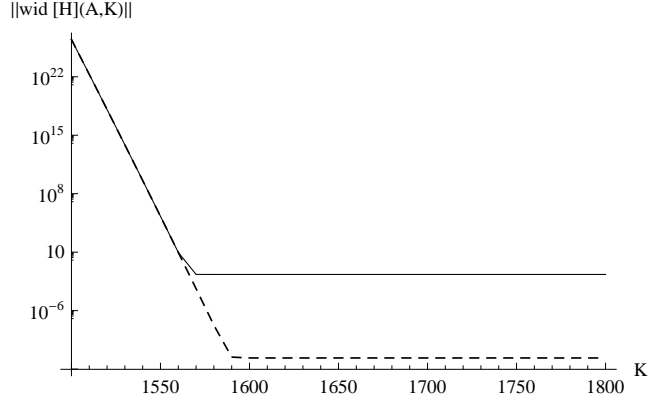}
		\caption{\label{fig:fig1}Plots of $||\iwid [H](A,K)||$ w.r.t. $K$, for two different precisions: Plain line for $p=110$ digits precision, and dashed line for $p=120$ digits precision.}
	\end{center}
\end{figure}

The interval scaling and squaring formula gives rise to $||\iwid [S](A,12,12)||\approx 7.2\times 10^{-6}$ computed using the standard double precision arithmetic. This can be improved using a decomposition $A=PMP^{-1}$ where $M$ is easier to exponentiate. Then, one can compute $\exp A=P\exp(P^{-1}AP)P^{-1}$ (where $P^{-1}$ has to be rigorously enclosed in order to maintain the rigorousness of the process). Using the Shur-decomposition, we obtain
\begin{equation}
	||\iwid P\,[S](P^{-1}AP,12,12)\,P^{-1}||\approx 7.2\times 10^{-11}.
\end{equation}

\subsection{Interval exponentiation of an interval matrix}\label{ss:detailed2}

In order to compare the different methods, we will use $0.1 A$ which is simpler to exponentiate. We have exponentiated $[A_\epsilon]:=0.1A+[-\epsilon,\epsilon]$ for various values of $\epsilon$ inside $[10^{-16},1]$ and the results are plotted on Figure \ref{fig:loglog}.
\begin{figure}[t!]
	\begin{center}
		\includegraphics[scale=0.85]{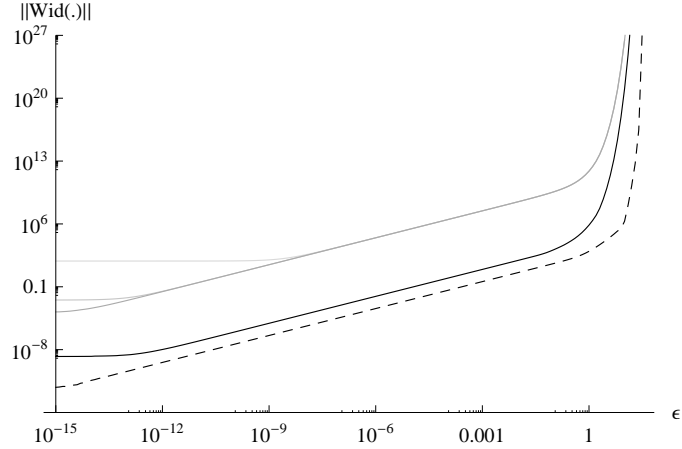}
		\caption{Log-log graphics for the comparison of the quality for different enclosing methods.\label{fig:loglog}}
	\end{center}
\end{figure}
The three plain gray curves represent $||\iwid[H]([A_\epsilon],K)||$ for $K=150$, $K=160$ and $K=170$ (increasing $K$ improves the enclosure, until $K=170$ above which no significantly improvement is shown). The black curve represents $||\iwid[S]([A_\epsilon],10,10)||$. The dashed line represents
\begin{equation}
	||\iwid\hull\{\exp \lb{A}_\epsilon,\exp \ub{A}_\epsilon\}||,
\end{equation}
which is a lower bound of $||\iwid\hull\exp([A_\epsilon])||$. Each plot show three phases: The first phase shows flat plots, then $||\iwid(\,\cdot\,) ||$ increases linearly\footnote{The linear plot displayed within the log-log scale indicates a polynomial behavior, the polynomial degree being fixed by the slope in the log-log representation. Here, the slope is one inside the log-log plot and thus so is the degree of the polynomial.} w.r.t. $\epsilon$ before eventually exponentially increasing.

During the flat phase, the rounding errors represent the main contribution to the final width of the enclosure. Thus, decreasing $\epsilon$ does not decrease the width of the final enclosure.

The linear phase is the most interesting. For these values of $\epsilon$, the width of $\hull\exp([A_\epsilon])$ growth linearly because the contribution of quadratic terms are negligible. On the other hand, the interval evaluations $[H]([A_\epsilon],K)$ and $[S]([A_\epsilon],10,10)$ are pessimistic, but it is well known that the pessimism of interval evaluation grows linearly w.r.t. the width of the interval arguments. Thus, the computed enclosure show a linear growth w.r.t. $\epsilon$ which are approximately
\begin{eqnarray}
	||\iwid[H]([A_\epsilon],K)|| & \approx & 1.17\times10^{-4} + 2.86\times 10^{10} \, \epsilon
	\\ ||\iwid[S]([A_\epsilon],10,10)|| & \approx & 1.80\times10^{-9} + 8.59\times 10^3 \, \epsilon.
\end{eqnarray}
This cleary shows how smaller is the pessimism introduced by the interval scaling and squaring process.

Finally, both the interval Horner evaluation and the interval scaling and squaring process show an exponential growth when $\epsilon$ is too large. The lower bound represented by the dashed line also shows a exponential growth, which proves that this is inherent to the exponentiation of an interval matrix. For such $\epsilon$, some matrices inside $[A_\epsilon]$ eventually see some of their eigenvalues becoming positive, leading to some exponential divergence of the underlying dynamical system, which is also observed on the matrices exponential.


\section*{Ackowledgements}

The author would like to thank the University of Central Arkansas, USA, who partially funded this work, and in particular Doctor Chenyi Hu for his helpful comments.



\end{document}